\documentclass{article}
\usepackage{amssymb,amsmath,amsfonts,amssymb}
\usepackage{amssymb,mathrsfs,bm}
\usepackage[round,colon]{natbib}
\usepackage{graphicx}
\usepackage[all]{xy}
\usepackage{anysize}
\marginsize{2.5cm}{2cm}{1.5cm}{4.5cm}

\newtheorem{theorem}{Theorem}

\newtheorem{corollary}[theorem]{Corollary}

\newtheorem{definition}[theorem]{Definition}
\newtheorem{example}[theorem]{Example}

\newtheorem{lemma}[theorem]{Lemma}

\numberwithin{equation}{section} \numberwithin{theorem}{section}
\newenvironment{proof}[1][Proof]{\noindent\textbf{#1.} }{\ \rule{0.5em}{0.5em}}

\begin{document}


\vspace{2cm}

{\textbf{\Large{This article has been accepted
and is scheduled for publication in JAP on March 2015}}}

\vspace{3cm}

\textbf{To cite this article}:
Nuria Torrado and Subhash C. Kochar (2015),
Stochastic order relations among parallel systems from Weibull distributions,
to appear in Journal of Applied Probability.

\vspace{3cm}

Thanks


\newpage

\title{Stochastic order relations among parallel systems \\ from Weibull
distributions}
\author{ Nuria Torrado \\
\small{Centre for Mathematics, University of Coimbra}\\
\small{Apartado 3008, EC Santa Cruz, 3001-501 Coimbra, Portugal} \\ \\
Subhash C. Kochar\\
\small{Fariborz Maseeh Department of Mathematics and Statistics}\\
\small{Portland State University, Portland, OR 97006, USA}}
\maketitle

\begin{abstract}
Let $X_{\lambda _{1}},X_{\lambda _{2}},\ldots ,X_{\lambda _{n}}$ be
independent Weibull random variables with $X_{\lambda _{i}}\sim
W(\alpha,\lambda _{i})$ where $\lambda _{i}>0$, for $i=1,\ldots ,n$. Let $%
X_{n:n}^{\lambda }$ denote the lifetime of the parallel system formed from $%
X_{\lambda _{1}},X_{\lambda _{2}},\ldots ,X_{\lambda _{n}}$. We investigate
the effect of the changes in the scale parameters $(\lambda_1, \ldots,
\lambda_n)$ on the magnitude of $X_{n:n}^{\lambda }$ according to
reverse hazard rate and likelihood ratio orderings.


\end{abstract}

\bigskip


\section{Introduction}

There is an extensive literature on stochastic orderings among order
statistics and spacings when the observations follow the exponential
distribution with different scale parameters, see for instance, \cite%
{Dykstra:1997, Bon:1999, Khaledi:2000, Wen:2007, Zhao:2009, Torrado:2013}
and the references therein. Also see a review paper by \cite{Kochar:2012} on
this topic. A natural way to extend these works is to consider random
variables with Weibull distributions since it includes exponential
distributions.

Let $X_{\lambda _{1}},X_{\lambda _{2}},\ldots ,X_{\lambda _{n}}$ be
independent Weibull random variables with $X_{\lambda _{i}}\sim W(\alpha
,\lambda _{i})$, $i=1,\ldots ,n$, where $\lambda _{i}>0,i=1,\ldots ,n$,
i.e., with density function%
\begin{equation*}
f_{i}(t)=\alpha \lambda _{i}\left( \lambda _{i}t\right) ^{\alpha
-1}e^{-\left( \lambda _{i}t\right) ^{\alpha }},\text{ }t>0.
\end{equation*}%
Let $h_{i}$ and $r_{i}$\ be the hazard rate and the reverse hazard rate
functions of $X_{\lambda _{i}}$, respectively. We denote by $%
X_{n:n}^{\lambda }$ the lifetime of the parallel system formed from $%
X_{\lambda _{1}},X_{\lambda _{2}},\ldots ,X_{\lambda _{n}}$. Then, its
distribution function is given by%
\begin{equation*}
F_{n:n}^{\lambda }(t)=\prod\limits_{i=1}^{n}F_{i}(t),
\end{equation*}%
its density function is%
\begin{equation*}
f_{n:n}^{\lambda }(t)=\prod\limits_{i=1}^{n}F_{i}(t)\sum_{i=1}^{n}r_{i}(t),
\end{equation*}%
and its reverse hazard rate function is%
\begin{equation}
r_{n:n}^{\lambda }(t)=\sum_{i=1}^{n}r_{i}(t)=\sum_{i=1}^{n}\frac{\alpha
\lambda _{i}\left( \lambda _{i}t\right) ^{\alpha -1}}{e^{\left( \lambda
_{i}t\right) ^{\alpha }}-1}.  \label{ecrh}
\end{equation}

For $0<\alpha \leq 1$, \cite{Khaledi:2006} proved that order statistics from
Weibull distributions with a common shape parameter and with scale
parameters as $(\lambda _{1},\ldots ,\lambda _{n})$ and $(\theta_1, \ldots,
\theta_n)$ are ordered in the usual stochastic order if one vector of scale
parameters majorizes the other one. For the proportional hazard rate model,
they also investigated the hazard rate and the dispersive orders among
parallel systems of a set of $n$ independent and non-identically distributed
random variables with that corresponding to a set of $n$ independent and
identically distributed random variables. Similar results for Weibull
distributions are also obtained by \cite{Fang:2012}.

In this article, we focus on stochastic orders to compare the magnitudes of
two parallel systems from Weibull distributions when one set of scale
parameters majorizes the other. The new results obtained here extend some of
those proved by \cite{Dykstra:1997} and \cite{Joo:2010} from exponential to
Weibull distributions. Also, we present some results for parallel systems
from multiple-outlier Weibull models.

The rest of the paper is organized as follows. In Section \ref{sec:def}, we
give the required definitions. We present some useful lemmas in Section \ref%
{sec:pre} which are used throughout the paper . In the last section  we
establish some new results on likelihood ratio ordering among parallel
systems from Weibull distributions.


\section{Basic definitions}

\label{sec:def}

In this section, we review some definitions and well-known notions of
stochastic orders and majorization concepts. Throughout this article \emph{%
increasing}\ means \emph{non-decreasing}\ and \emph{decreasing}\ means \emph{%
non-increasing}.

Let $X$ and $Y$ be univariate random variables with cumulative distribution
functions (c.d.f.'s) $F$ and $G$, survival functions $\Bar{F}\left(= 1 -
F\right)$ and $\Bar{G}\left(= 1 - G\right)$, p.d.f.'s $f$ and $g$, hazard
rate functions $h_{F}\left(=f/\,\Bar{F}\right)$ and $h_{G}\left(=g/\,\Bar{G}%
\right)$, and reverse hazard rate functions $r_{F}\left(=f/F\right)$ and $%
r_{G}\left(=g/G\right)$, respectively. The following definitions introduce
stochastic orders, which are considered in this article, to compare the
magnitudes of two random variables. For more details on stochastic
comparisons, see \cite{Shaked:2007}.

\begin{definition}
We say that $X$ is smaller than $Y$ in the:

\begin{itemize}
\item[a)] usual stochastic order, denoted by $X\leq_{st}Y$ or $F\leq_{st}G$,
if $\Bar{F}(t)\leq\Bar{G}(t)$ for all $t$,

\item[b)] hazard rate order, denoted by $X\leq_{hr}Y$ or $F\leq_{hr}G$, if $%
\Bar{G}(t)/\,\Bar{F}(t)$ is increasing in $t$ for all $t$ for which this
ratio is well defined,

\item[c)] reverse hazard rate order, denoted by $X\leq_{rh}Y$ or $%
F\leq_{rh}G $, if $G(t)/F(t)$ is increasing in $t$ for all $t$ for which the
ratio is well defined,

\item[d)] likelihood ratio order, denoted by $X\leq_{lr}Y$ or $F\leq_{lr}G$,
if $g(t)/f(t)$ is increasing in $t$ for all $t$ for which the ratio is well
defined.
\end{itemize}
\end{definition}

In this paper, we shall be using the following Theorem 1.C.4 of Shaked and
Shanthikumar (2007).

\begin{description}
\item {(a)} If $X\leq_{hr} Y$ and if $\frac{h_Y(t)}{h_X(t)}$ increases in $t$%
, then $X\leq_{lr} Y$.

\item {(b)} If $X\leq_{rh} Y$ and if $\frac{r_Y(t)}{r_X(t)}$ increases in $t$%
, then $X\leq_{lr} Y$.
\end{description}

We shall also be using the concept of majorization in our discussion. Let $%
\{ x_{(1)}, x_{(2)},\ldots, x_{(n)} \}$ denote the increasing arrangement of
the components of the vector ${\bm x}=\left(x_{1},x_{2},\ldots,x_{n}\right)$.

\begin{definition}
The vector ${\bm x}$ is said to be majorized by the vector ${\bm y}$,
denoted by ${\bm x}\leq^{m}{\bm y}$, if
\begin{equation*}
\sum_{i=1}^{j}x_{(i)}\geq\sum_{i=1}^{j}y_{(i)},\quad\text{for }
j=1,\ldots,n-1 \quad \text{ and } \quad
\sum_{i=1}^{n}x_{(i)}=\sum_{i=1}^{n}y_{(i)}.
\end{equation*}
\end{definition}

Functions that preserve the ordering of majorization are said to be
Schur-convex as defined below.

\begin{definition}
\label{ch0:def:04} A real valued function $\varphi$ defined on a set $%
\mathcal{A}\in\Re^{n}$ is said to be \emph{Schur-convex} (\emph{Schur-concave%
}) on $\mathcal{A}$ if
\begin{equation*}
{\bm x} \leq^{m}{\bm y}\, \text{on }\,\mathcal{A}\Rightarrow \varphi({\bm x}
)\leq (\geq) \varphi({\bm y} ).
\end{equation*}
\end{definition}

A concept of weak majorization is the following.

\begin{definition}
The vector ${\bm x}$ is said to be weakly majorized by the vector ${\bm y}$,
denoted by ${\bm x}\leq^{w}{\bm y}$, if
\begin{equation*}
\sum_{i=1}^{j}x_{(i)}\geq\sum_{i=1}^{j}y_{(i)},\quad\text{for } j=1,\ldots,n.
\end{equation*}
\end{definition}

It is known that ${\bm x}\leq^{m}{\bm y}\Rightarrow {\bm x}\leq^{w}{\bm y}$.
The converse is, however, not true. For extensive and comprehensive details
on the theory of majorization orders and their applications, please refer to
the book of \cite{Marshall:1979}.


\section{Preliminaries results}

\label{sec:pre}

In this section, we first preset several useful lemmas which will be used in
the next section to prove our main results.

\begin{lemma}
\label{lemU} For $t\geq 0$, the function
\begin{equation}
u(t)=\frac{t^{\alpha }}{e^{t^{\alpha }}-1}  \label{eclemU}
\end{equation}%
is decreasing for any $\alpha >0$ and convex for $0<\alpha \leq 1$.
\end{lemma}

\begin{proof}
Note that, for $t\geq 0$, $u(t)=\psi _{1}(t^{\alpha })$ with%
\begin{equation*}
\psi _{1}(t)=\frac{t}{e^{t}-1}.
\end{equation*}%
It is easy to check that $\psi _{1}(t)$ is a decreasing and convex function.
Therefore,%
\begin{equation*}
\frac{d}{dt}u(t)=\alpha t^{\alpha -1}\frac{d}{dt}\psi _{1}(t)\leq 0
\end{equation*}%
for any $\alpha >0$, and
\begin{equation*}
\frac{d^{2}}{dt^{2}}u(t)=\alpha (\alpha -1)t^{\alpha -2}\frac{d}{dt}\psi
_{1}(t)+\left( \alpha t^{\alpha -1}\right) ^{2}\frac{d^{2}}{dt^{2}}\psi
_{1}(t)\geq 0,
\end{equation*}%
since $0<\alpha \leq 1$. Hence $u(t)$ is decreasing for any $\alpha >0$ and
convex for $0<\alpha \leq 1$ in $[0,\infty ).$
\end{proof}

\begin{lemma}
\label{lemV} For $t\geq 0$, the function
\begin{equation}
v(t)=\frac{t^{\alpha }}{1-e^{-t^{\alpha }}}  \label{eclemV}
\end{equation}%
is increasing for any $\alpha >0$.
\end{lemma}

\begin{proof}
Note that, for $t\geq 0$, $v(t)=\psi _{2}(t^{\alpha })$ with%
\begin{equation*}
\psi _{2}(t)=\frac{t}{1-e^{-t}}.
\end{equation*}%
As shown by \cite{Khaledi:2000} in Lemma 2.1, $\psi _{2}(t)$\ is increasing
in $t\geq 0 $. Hence $v(t)$ is also increasing because%
\begin{equation*}
\frac{d}{dt}v(t)=\alpha t^{\alpha -1}\frac{d}{dt}\psi _{2}(t)\geq 0\text{,}
\end{equation*}%
for any $\alpha >0$.
\end{proof}

\begin{lemma}
\label{lem02}For $t\geq 0$, the function%
\begin{equation}
\frac{te^{t}(1+t-e^{t})}{(e^{t}-1)^{3}}  \label{eclem}
\end{equation}%
is increasing.
\end{lemma}

\begin{proof}
The derivative of (\ref{eclem}) is, for $t\geq 0$,%
\begin{equation*}
\frac{e^{t}}{(e^{t}-1)^{4}}(-1+e^{2t}(t-1)-t(3+t)+e^{t}(2-2t(t-1))).
\end{equation*}%
Thus, it is sufficient to prove that%
\begin{equation*}
M_{1}(t)=-1+e^{2t}(t-1)-t(3+t)+e^{t}(2-2t(t-1))\geq 0.
\end{equation*}%
Since $M_{1}(0)=0$, we have to prove that
\begin{equation*}
M_{1}^{\prime }(t)=e^{2t}(2t-1)-3-20+2e^{t}(2-t(1+t))\geq 0.
\end{equation*}%
Again, $M_{1}^{\prime }(0)=0$, so we have to prove that%
\begin{equation*}
M_{1}^{\prime \prime }(t)=2(-1+e^{t}(1-t(3-2e^{t}+t)))\geq 0.
\end{equation*}%
Denote, for $t\geq 0$,%
\begin{equation*}
M_{2}(t)=-1+e^{t}(1-t(3-2e^{t}+t)).
\end{equation*}%
Since the derivative of $M_{2}(t)$ is,%
\begin{equation*}
M_{2}^{\prime }(t)=e^{t}(-2-t(5+t)+e^{t}(2+4t))
\end{equation*}%
and $e^{t}\geq 1+t$, it follows that $M_{2}^{\prime }(t)\geq 0$, because%
\begin{equation*}
-2-t(5+t)+e^{t}(2+4t)\geq -2-t(5+t)+(1+t)(2+4t)=t(1+3t)\geq 0.
\end{equation*}%
That is, $M_{2}(t)$ is increasing in $t\geq 0$. Observing that $M_{2}(0)=0$,
we have $M_{2}(t)\geq 0$ for $t\geq 0$. The required result follows
immediately.
\end{proof}

\begin{lemma}
\label{lemW}For $t\geq 0$, the function%
\begin{equation}
w(t)=\frac{\alpha t^{2\alpha -1}e^{t^{\alpha }}(1+t^{\alpha }-e^{t^{\alpha
}})}{(e^{t^{\alpha }}-1)^{3}}  \label{eclemW}
\end{equation}%
is increasing for $0<\alpha \leq 1$.
\end{lemma}

\begin{proof}
Note that, for $t\geq 0$,
\begin{equation*}
w(t)=\alpha t^{\alpha -1}\psi _{3}(t^{\alpha }),
\end{equation*}%
with%
\begin{equation*}
\psi _{3}(t)=\frac{te^{t}(1+t-e^{t})}{(e^{t}-1)^{3}}.
\end{equation*}%
From Lemma \ref{lem02} we know that $\psi _{3}(t)$ is an increasing
function, then
\begin{equation*}
\frac{d}{dt}w(t)=\alpha (\alpha -1)t^{\alpha -2}\psi _{3}(t^{\alpha
})+\left( \alpha t^{\alpha -1}\right) ^{2}\frac{d}{dt}\psi _{3}(t)\geq 0,
\end{equation*}%
since $0<\alpha \leq 1$ and $\psi _{3}(t)\leq 0$ because $e^{t}\geq 1+t$ for
$t\geq 0$. Hence $w(t)$ is increasing in $[0,\infty )$ for $0<\alpha \leq 1.$
\end{proof}



\section{Main results}
\label{sec:parallel}

In this section, we establish likelihood ratio ordering between parallel
systems based on two sets of heterogeneous Weibull random variables with a
common shape parameter and with scale parameters which are ordered according
to a majorization order. First, we establish a comparison among parallel
systems according to reverse hazard rate ordering when the common shape
parameter $\alpha$ satisfies $0<\alpha \le 1 $.

\begin{theorem}
\label{th07} Let $X_{\lambda _{1}},X_{\lambda _{2}},\ldots ,X_{\lambda _{n}}$
be independent random variables with $X_{\lambda _{i}}\sim W(\alpha ,\lambda
_{i})$, where $\lambda _{i}>0$, $i=1,\ldots ,n$, and let $X_{\theta
_{1}},X_{\theta _{2}},\ldots ,X_{\theta _{n}}$ be another set of independent
random variables with $X_{\theta _{i}}\sim W(\alpha ,\theta _{i})$, where $%
\theta _{i}>0$, $i=1,\ldots ,n$. Then for $0<\alpha \leq 1$,%
\begin{equation*}
\left( \lambda _{1},\ldots ,\lambda _{n}\right) \preceq ^{w}\left( \theta
_{1},\ldots ,\theta _{n}\right) \Rightarrow X_{n:n}^{\lambda }\leq
_{rh}X_{n:n}^{\theta }.
\end{equation*}
\end{theorem}

\begin{proof}
Fix $t\geq 0$. Then the reverse hazard rate of $X_{n:n}$ is%
\begin{equation*}
r_{n:n}^{\lambda }(t)=\frac{\alpha }{t}\sum_{i=1}^{n}\frac{\left( \lambda
_{i}t\right) ^{\alpha }}{e^{\left( \lambda _{i}t\right) ^{\alpha }}-1}=\frac{%
\alpha }{t}\sum_{i=1}^{n}u(\lambda _{i}t),
\end{equation*}%
where $u(t)$ is defined as in (\ref{eclemU}). From Theorem A.8 of \cite%
{Marshall:1979} (p.59) it suffices to show that, for each $t\geq 0$, $%
r_{n:n}^{\lambda }(t)$ is decreasing in each $\lambda _{i}$, $i=1,\ldots ,n$%
, and is a Schur-convex function of $\left( \lambda _{1},\ldots ,\lambda
_{n}\right) $. It is well known that the hazard rate of the Weibull
distribution is decreasing in $t\geq 0$ when $0<\alpha \leq 1$ (see \cite%
{Marshall:2007}, p. 324), and therefore, its reverse hazard rate function is
also decreasing. Clearly, from (\ref{ecrh}), the reverse hazard rate
function of $X_{n:n}$ is decreasing in each $\lambda _{i}$. Now, from
Proposition C.1 of \cite{Marshall:1979} (p. 64), in order to establish the
Schur-convexity of $r_{n:n}^{\lambda }(t)$, it is enough to prove the
convexity of $u(t)$. Note that, from Lemma \ref{lemU}, we know that $u(t)$
is a convex function for $0<\alpha \leq 1$. Hence $r_{n:n}^{\lambda }(t)$ is
a Schur-convex function of $\left( \lambda _{1},\ldots ,\lambda _{n}\right) $%
.
\end{proof}

\vskip .5em Since ${\bm x}\leq ^{m}{\bm y}\Rightarrow {\bm x}\leq ^{w}{\bm y}
$, the following corollary follows immediately from Theorem \ref{th07}.

\begin{corollary}
\label{cor1} Let $X_{\lambda _{1}},X_{\lambda _{2}},\ldots ,X_{\lambda _{n}}$
be independent random variables with $X_{\lambda _{i}}\sim W(\alpha ,\lambda
_{i})$, where $\lambda _{i}>0$, $i=1,\ldots ,n$, and let $X_{\theta
_{1}},X_{\theta _{2}},\ldots ,X_{\theta _{n}}$ be another set of independent
random variables with $X_{\theta _{i}}\sim W(\alpha ,\theta _{i})$, where $%
\theta _{i}>0$, $i=1,\ldots ,n$. Then for $0<\alpha \leq 1$,
\begin{equation*}
\left( \lambda _{1},\ldots ,\lambda _{n}\right) \preceq ^{m}\left( \theta
_{1},\ldots ,\theta _{n}\right) \Rightarrow X_{n:n}^{\lambda }\leq
_{rh}X_{n:n}^{\theta }.
\end{equation*}
\end{corollary}

Note that Corollary \ref{cor1} was proved by \cite{Khaledi:2011} for
generalized gamma distribution when $p=q<1$ which corresponds to Weibull
distribution with shape parameter $\alpha <1$.

A natural question is whether the results of Theorem \ref{th07} and
Corollary \ref{cor1} can be strengthened from reverse hazard rate ordering
to likelihood ratio ordering. First we consider the case when $n=2$.

\begin{theorem}
\label{th08} Let $X_{\lambda _{1}},X_{\lambda _{2}}$ be independent random
variables with $X_{\lambda _{i}}\sim W(\alpha ,\lambda _{i})$ where $\lambda
_{i}>0$, $i=1,2$, and let $X_{\theta _{1}},X_{\theta _{2}}$ be independent
random variables with $X_{\theta _{i}}\sim W(\alpha ,\theta _{i})$ where $%
\theta _{i}>0$, $i=1,2$. Then%
\begin{equation*}
\left( \lambda _{1},\lambda _{2}\right) \preceq ^{m}\left( \theta
_{1},\theta _{2}\right) \Rightarrow \frac{r_{2:2}^{\theta }(t)}{%
r_{2:2}^{\lambda }(t)}\text{ is increasing in }t\text{, for }0<\alpha \leq 1.
\end{equation*}
\end{theorem}

\begin{proof}
From (\ref{ecrh}) we have%
\begin{equation*}
r_{2:2}^{\theta }(t)=\frac{\alpha }{t}\left( \frac{\left( \theta
_{1}t\right) ^{\alpha }}{e^{(\theta _{1}t)^{\alpha }}-1}+\frac{\left( \theta
_{2}t\right) ^{\alpha }}{e^{(\theta _{2}t)^{\alpha }}-1}\right) ,
\end{equation*}%
then%
\begin{equation*}
\phi (t)=\frac{r_{2:2}^{\theta }(t)}{r_{2:2}^{\lambda }(t)}=\frac{u(\theta
_{1}t)+u(\theta _{2}t)}{u(\lambda _{1}t)+u(\lambda _{2}t)},
\end{equation*}%
where $u(t)$ is the function defined in (\ref{eclemU}). Note that the
derivative of $u(t)$ with respect to $t$ is%
\begin{equation*}
u^{\prime }(t)=\frac{\alpha t^{\alpha -1}\left( e^{t^{\alpha }}-1-t^{\alpha
}e^{t^{\alpha }}\right) }{\left( e^{t^{\alpha }}-1\right) ^{2}}=\frac{\alpha
}{t}u(t)s(t),
\end{equation*}%
where%
\begin{eqnarray*}
s(t) &=&\frac{-1+e^{t^{\alpha }}(1-t^{\alpha })}{e^{t^{\alpha }}-1}=1-\frac{%
t^{\alpha }e^{t^{\alpha }}}{e^{t^{\alpha }}-1} \\
&=&1-\frac{t^{\alpha }}{1-e^{-t^{\alpha }}}=1-v(t),
\end{eqnarray*}%
with $v(t)$ defined as in (\ref{eclemV}). Therefore, from Lemma \ref{eclemV}%
, we know that$\ s(t)$ is a decreasing function. We have to show that $\phi
^{\prime }(t)\geq 0$ for all $t\geq 0$. The derivative of $\phi (t)$ is, for
$t\geq 0$,
\begin{eqnarray*}
\phi ^{\prime }(t)&\overset{\text{sign}}{=}&\left( u(\theta _{1}t)s(\theta
_{1}t)+u(\theta _{2}t)s(\theta _{2}t)\right) \left( u(\lambda
_{1}t)+u(\lambda _{2}t)\right)  \\
&&-\left( u(\theta _{1}t)+u(\theta _{2}t)\right) \left( u(\lambda
_{1}t)s(\lambda _{1}t)+u(\lambda _{2}t)s(\lambda _{2}t)\right) .
\end{eqnarray*}%
Thus, we have to prove that the function%
\begin{equation*}
L\left( \theta _{1},\theta _{2}\right) =\frac{pu(\theta _{1}t)s(\theta
_{1}t)+qu(\theta _{2}t)s(\theta _{2}t)}{pu(\theta _{1}t)+qu(\theta _{2}t)}
\end{equation*}%
is Schur-convex in $\left( \theta _{1},\theta _{2}\right) $. On
differentiating $L\left( \theta _{1},\theta _{2}\right) $ with respect to $%
\theta _{1}$, we get%
\begin{eqnarray}
\frac{dL\left( \theta _{1},\theta _{2}\right) }{d\theta _{1}}&\overset{%
\text{sign}}{=}&\left[ u^{\prime }(\theta _{1}t)s(\theta _{1}t)+u(\theta
_{1}t)s^{\prime }(\theta _{1}t)\right] \left[ u(\theta _{1}t)+u(\theta _{2}t)%
\right]   \label{ec05} \\
&&-\left[ u(\theta _{1}t)s(\theta _{1}t)+u(\theta _{2}t)s(\theta _{2}t)%
\right] u^{\prime }(\theta _{1}t)  \notag \\
&=&u(\theta _{2}t)u^{\prime }(\theta _{1}t)\left[ s(\theta _{1}t)-s(\theta
_{2}t)\right] +u(\theta _{1}t)s^{\prime }(\theta _{1}t)\left[ u(\theta
_{1}t)+u(\theta _{2}t)\right] .  \notag
\end{eqnarray}%
Note that $u(t)s^{\prime }(t)=w(t)$ which is defined in (\ref{eclem}) and
from Lemma \ref{lem02}, we know that $w(t)$ is an increasing function for $%
0<\alpha \leq 1$. Then, (\ref{ec05}) can be rewritten as
\begin{equation*}
\frac{dL\left( \theta _{1},\theta _{2}\right) }{d\theta _{1}}\overset{\text{%
sign}}{=} u(\theta _{2}t)u^{\prime }(\theta _{1}t)\left[ s(\theta
_{1}t)-s(\theta _{2}t)\right] +w(\theta _{1}t)\left[ u(\theta
_{1}t)+u(\theta _{2}t)\right] .
\end{equation*}%
By interchanging $\theta _{1}$ and $\theta _{2}$, we have%
\begin{equation*}
\frac{dL\left( \theta _{1},\theta _{2}\right) }{d\theta _{2}}\overset{\text{%
sign}}{=}u(\theta _{1}t)u^{\prime }(\theta _{2}t)\left[ s(\theta
_{2}t)-s(\theta _{1}t)\right] +w(\theta _{2}t)\left[ u(\theta
_{1}t)+u(\theta _{2}t)\right] .
\end{equation*}%
Thus,%
\begin{eqnarray*}
\frac{dL\left( \theta _{1},\theta _{2}\right) }{d\theta _{1}}-\frac{%
dL\left( \theta _{1},\theta _{2}\right) }{d\theta _{2}}&\overset{\text{sign}}{%
=}& \left[ s(\theta _{1}t)-s(\theta _{2}t)\right] \left[ u(\theta
_{2}t)u^{\prime }(\theta _{1}t)+u(\theta _{1}t)u^{\prime }(\theta _{2}t)%
\right] + \\
&&\left[ u(\theta _{1}t)+u(\theta _{2}t)\right] \left[ w(\theta
_{1}t)-w(\theta _{2}t)\right]  \\
&\leq &0,
\end{eqnarray*}%
if $\theta _{1}\leq \theta _{2}$, since $s(t)$ is decreasing, $u^{\prime
}(t)\leq 0$ because $u(t)$ is a decreasing function and $w(t)$ is an
increasing function. Hence,%
\begin{equation*}
\left( \theta _{1}-\theta _{2}\right) \left( \frac{dL\left( \theta
_{1},\theta _{2}\right) }{d\theta _{1}}-\frac{dL\left( \theta _{1},\theta
_{2}\right) }{d\theta _{2}}\right) \geq 0.
\end{equation*}
\end{proof}


In the next result, we extend Theorem \ref{th07} from reverse hazard rate
ordering to likelihood ratio ordering for $n=2$.

\begin{theorem}
\label{th09} Let $X_{\lambda _{1}},X_{\lambda _{2}}$ be independent random
variables with $X_{\lambda _{i}}\sim W(\alpha ,\lambda _{i})$ where $\lambda
_{i}>0$, $i=1,2$, and let $X_{\theta _{1}},X_{\theta _{2}}$ be independent
random variables with $X_{\theta _{i}}\sim W(\alpha ,\theta _{i})$ where $%
\theta _{i}>0$, $i=1,2$. Then for $0<\alpha \leq 1$,
\begin{equation*}
\left( \lambda _{1},\lambda _{2}\right) \preceq ^{m}\left( \theta
_{1},\theta _{2}\right) \Rightarrow X_{2:2}^{\lambda }\leq
_{lr}X_{2:2}^{\theta }.
\end{equation*}
\end{theorem}

\begin{proof}
The required result follows from Theorem 1.C.4 in \cite{Shaked:2007} and
Theorems \ref{th07} and \ref{th08}.
\end{proof}

Note that Theorem \ref{th09} generalizes and strengthens Theorem 3.1 of \cite%
{Dykstra:1997} from exponential to Weibull distributions.

One may wonder whether one can extend Theorem \ref{th09}\ for $\alpha >1$.
The following example gives a negative answer.

\begin{example}
Let $\left( X_{\lambda _{1}},X_{\lambda _{2}}\right) $ be a vector of
heterogeneous Weibull random variables, $W(\alpha ,\lambda _{i})$, with $%
\alpha =2$ and scale parameter vector $\left( 1.5,2\right) $. Let $\left(
X_{\theta _{1}},X_{\theta _{2}}\right) $ be a vector of heterogeneous
Weibull random variables, $W(\alpha ,\theta _{i})$, with $\alpha =2$ and
scale parameter vector $\left( 1,2.5\right) $. Obviously $\left( \lambda
_{1},\lambda _{2}\right) \preceq ^{m}\left( \theta _{1},\theta _{2}\right) $%
, however $X_{2:2}^{\lambda }\nleq _{lr}X_{2:2}^{\theta }$ since $%
f_{2:2}^{\theta }(t)/f_{2:2}^{\lambda }(t)$ is not increasing in $t$ as it
can be seen from Figure \ref{figcocpdfWeib2}.
\end{example}

\begin{figure}[!th]
\centering
\includegraphics[scale=0.7, keepaspectratio]{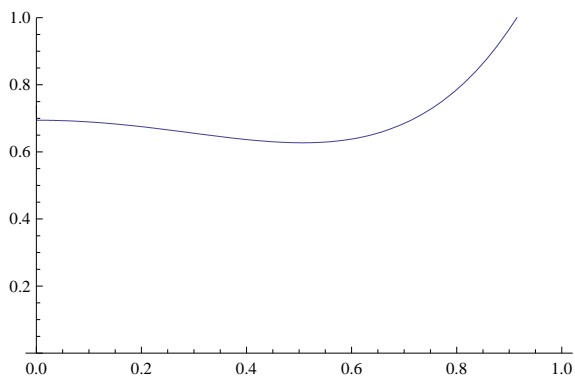}
\caption{Plot of the $f_{n:n}^{\protect\theta }(t)/f_{n:n}^{\protect\lambda %
}(t)$ when $\protect\alpha =2$, $\protect\theta =\left( 1,2.5\right) $ and $%
\protect\lambda =\left( 1.5,2\right) $ for random variables with Weibull
distributions}
\label{figcocpdfWeib2}
\end{figure}


For comparing the lifetimes of two parallel systems with independent Weibull
components, $W(\alpha ,\lambda _{i})$, we have proved in Theorem \ref{th09}
that they are ordered according to likelihood ratio ordering under the
condition of majorization order with respect to $(\lambda _{1},\lambda _{2})$
when $0<\alpha \leq 1$. For $n >2$, the problem is still open. However, in
the multiple-outlier Weibull model, a similar result still holds. But first,
we need to prove the following result.

\begin{theorem}
\label{th15} Let $X_{{1}},\ldots ,X_{{n}}$ be independent
random variables following the multiple-outlier Weibull model such that $%
X_{{i}}\sim W(\alpha ,\lambda _{1})$ for $i=1,\ldots ,p$\ and $%
X_{{j}}\sim W(\alpha ,\lambda _{2})$ for $j=p+1,\ldots ,n$, with $%
\lambda _{1},\lambda _{2}>0$. Let $Y_{{1}},\ldots ,Y_{{n}}$
be another set of independent random variables following the
multiple-outlier Weibull model such that $Y_{{i}}\sim W(\alpha
,\theta _{1})$ for $i=1,\ldots ,p$\ and $Y_{{j}}\sim W(\alpha
,\theta _{2})$ for $j=p+1,\ldots ,n$, with $\theta _{1},\theta _{2}>0$. Then
for $0<\alpha \leq 1$,
\begin{equation*}
(\underbrace{\lambda _{1},\ldots ,\lambda _{1}}_{p},\underbrace{\lambda
_{2},\ldots ,\lambda _{2}}_{q})
\preceq ^{m}
(\underbrace{\theta _{1},\ldots ,\theta _{1}}_{p},\underbrace{\theta
_{2},\ldots ,\theta _{2}}_{q})
\Rightarrow \frac{r_{n:n}^{\theta }(t)}{%
r_{n:n}^{\lambda }(t)}\text{ is increasing in }t,
\end{equation*}
where $p+q=n$.
\end{theorem}

\begin{proof}
From (\ref{ecrh}) we have%
\begin{equation*}
r_{n:n}^{\theta }(t)=\frac{\alpha }{t}\left( p\frac{\left( \theta
_{1}t\right) ^{\alpha }}{e^{(\theta _{1}t)^{\alpha }}-1}+q\frac{\left(
\theta _{2}t\right) ^{\alpha }}{e^{(\theta _{2}t)^{\alpha }}-1}\right) ,
\end{equation*}%
where $p+q=n$, then%
\begin{equation*}
\phi (t)=\frac{r_{n:n}^{\theta }(t)}{r_{n:n}^{\lambda }(t)}=\frac{pu(\theta
_{1}t)+qu(\theta _{2}t)}{pu(\lambda _{1}t)+qu(\lambda _{2}t)},
\end{equation*}%
where $u(t)$ is the function defined in (\ref{eclemU}). We have to show that
$\phi ^{\prime }(t)\geq 0$ for all $t\geq 0$. The derivative of $\phi (t)$
is, for $t\geq 0$,
\begin{eqnarray*}
\phi ^{\prime }(t)&\overset{\text{sign}}{=}&\left( pu(\theta _{1}t)s(\theta
_{1}t)+qu(\theta _{2}t)s(\theta _{2}t)\right) \left( pu(\lambda
_{1}t)+qu(\lambda _{2}t)\right)  \\
&&-\left( pu(\theta _{1}t)+qu(\theta _{2}t)\right) \left( pu(\lambda
_{1}t)s(\lambda _{1}t)+qu(\lambda _{2}t)s(\lambda _{2}t)\right) ,
\end{eqnarray*}%
where $s(t)=1-v(t)$ and $v(t)$ is defined in (\ref{eclemV}). Thus, we have
to prove that the function%
\begin{equation*}
L\left( \theta _{1},\theta _{2}\right) =\frac{pu(\theta _{1}t)s(\theta
_{1}t)+qu(\theta _{2}t)s(\theta _{2}t)}{pu(\theta _{1}t)+qu(\theta _{2}t)}
\end{equation*}%
is Schur-convex in $\left( \theta _{1},\theta _{2}\right) $. On
differentiating $L\left( \theta _{1},\theta _{2}\right) $ with respect to $%
\theta _{1}$, we get%
\begin{eqnarray}
\frac{dL\left( \theta _{1},\theta _{2}\right) }{d\theta _{1}}&\overset{%
\text{sign}}{=}&\left[ u^{\prime }(\theta _{1}t)s(\theta _{1}t)+u(\theta
_{1}t)s^{\prime }(\theta _{1}t)\right] \left[ pu(\theta _{1}t)+qu(\theta
_{2}t)\right]   \label{ec04} \\
&&-\left[ pu(\theta _{1}t)s(\theta _{1}t)+qu(\theta _{2}t)s(\theta _{2}t)%
\right] u^{\prime }(\theta _{1}t)  \notag \\
&=&qu(\theta _{2}t)u^{\prime }(\theta _{1}t)\left[ s(\theta _{1}t)-s(\theta
_{2}t)\right] +u(\theta _{1}t)s^{\prime }(\theta _{1}t)\left[ pu(\theta
_{1}t)+qu(\theta _{2}t)\right] .  \notag
\end{eqnarray}%
Note that $u(t)s^{\prime }(t)=w(t)$ which is defined in (\ref{eclem}) and
from Lemma \ref{lem02}, we know that $w(t)$ is an increasing function for $%
0<\alpha \leq 1$. Then, (\ref{ec04}) can be rewritten as
\begin{equation*}
\frac{dL\left( \theta _{1},\theta _{2}\right) }{d\theta _{1}}\overset{\text{%
sign}}{=}qu(\theta _{2}t)u^{\prime }(\theta _{1}t)\left[ s(\theta
_{1}t)-s(\theta _{2}t)\right] +w(\theta _{1}t)\left[ pu(\theta
_{1}t)+qu(\theta _{2}t)\right] .
\end{equation*}%
By interchanging $\theta _{1}$ and $\theta _{2}$, we have%
\begin{equation*}
\frac{dL\left( \theta _{1},\theta _{2}\right) }{d\theta _{2}}\overset{\text{%
sign}}{=}pu(\theta _{1}t)u^{\prime }(\theta _{2}t)\left[ s(\theta
_{2}t)-s(\theta _{1}t)\right] +w(\theta _{2}t)\left[ pu(\theta
_{1}t)+qu(\theta _{2}t)\right] .
\end{equation*}%
Thus,%
\begin{eqnarray*}
\frac{dL\left( \theta _{1},\theta _{2}\right) }{d\theta _{1}}-\frac{%
dL\left( \theta _{1},\theta _{2}\right) }{d\theta _{2}}&\overset{\text{sign}}{%
=}&\left[ s(\theta _{1}t)-s(\theta _{2}t)\right] \left[ qu(\theta
_{2}t)u^{\prime }(\theta _{1}t)+pu(\theta _{1}t)u^{\prime }(\theta _{2}t)%
\right] + \\
&&\left[ pu(\theta _{1}t)+qu(\theta _{2}t)\right] \left[ w(\theta
_{1}t)-w(\theta _{2}t)\right]  \\
&\leq &0,
\end{eqnarray*}%
if $\theta _{1}\leq \theta _{2}$, since $s(t)$ is decreasing, $u^{\prime
}(t)\leq 0$ because $u(t)$ is a decreasing function and $w(t)$ is an
increasing function. Hence,%
\begin{equation*}
\left( \theta _{1}-\theta _{2}\right) \left( \frac{dL\left( \theta
_{1},\theta _{2}\right) }{d\theta _{1}}-\frac{dL\left( \theta _{1},\theta
_{2}\right) }{d\theta _{2}}\right) \geq 0.
\end{equation*}
\end{proof}

Using Theorems \ref{th07} and \ref{th15} and Theorem 1.C.4 in \cite%
{Shaked:2007}, we have the following result.

\begin{theorem}
\label{th16} Let $X_{{1}},\ldots ,X_{{n}}$ be independent
random variables following the multiple-outlier Weibull model such that $%
X_{{i}}\sim W(\alpha ,\lambda _{1})$ for $i=1,\ldots ,p$\ and $%
X_{{j}}\sim W(\alpha ,\lambda _{2})$ for $j=p+1,\ldots ,n$, with $%
\lambda _{1},\lambda _{2}>0$. Let $Y_{{1}},\ldots ,Y_{{n}}$
be another set of independent random variables following the
multiple-outlier Weibull model such that $Y_{{i}}\sim W(\alpha
,\theta _{1})$ for $i=1,\ldots ,p$\ and $Y_{{j}}\sim W(\alpha
,\theta _{2})$ for $j=p+1,\ldots ,n$, with $\theta _{1},\theta _{2}>0$. Then
for $0<\alpha \leq 1$,
\begin{equation*}
(\underbrace{\lambda _{1},\ldots ,\lambda _{1}}_{p},\underbrace{\lambda
_{2},\ldots ,\lambda _{2}}_{q})
\preceq ^{m}
(\underbrace{\theta _{1},\ldots ,\theta _{1}}_{p},\underbrace{\theta
_{2},\ldots ,\theta _{2}}_{q})
\Rightarrow X_{n:n}^{\lambda }\leq
_{lr}X_{n:n}^{\theta },
\end{equation*}
where $p+q=n$.
\end{theorem}


The above theorem says that the lifetime of a parallel system consisting of
two types of Weibull components with a common shape parameter between 0 and
1 is stochastically larger according to likelihood ratio ordering when the
scale parameters are more dispersed according to majorization.

In the following results, we investigate whether likelihood ratio ordering
holds among parallel systems when the scale parameters of the Weibull
distributions are ordered according to weak majorization order and the
common shape parameter $\alpha$ is arbitrary.

\begin{theorem}
\label{th10}Let $X_{\lambda _{1}},X_{\lambda }$ be independent random
variables with $X_{\lambda _{1}}\sim W(\alpha ,\lambda _{1})$ and $%
X_{\lambda }\sim W(\alpha ,\lambda )$ where $\lambda _{1},\lambda >0$. Let $%
Y_{\lambda _{1}^{\ast }},Y_{\lambda }$ be independent random variables with $%
Y_{\lambda _{1}^{\ast }}\sim W(\alpha ,\lambda _{1}^{\ast })$ and $%
Y_{\lambda }\sim W(\alpha ,\lambda )$, where $\lambda _{1}^{\ast },\lambda
>0 $. Suppose $\lambda _{1}^{\ast }=\min (\lambda ,\lambda _{1},\lambda
_{1}^{\ast })$, then for any $\alpha >0$,
\begin{equation*}
\left( \lambda _{1},\lambda \right) \preceq ^{w}\left( \lambda _{1}^{\ast
},\lambda \right) \Rightarrow \frac{r_{2:2}^{\ast }(t)}{r_{2:2}(t)}\text{ is
increasing in }t.
\end{equation*}%
.
\end{theorem}

\begin{proof}
From (\ref{ecrh}), the reverse hazard rate function of $X_{2:2}$ is
\begin{equation*}
r_{2:2}(t)=\frac{\alpha }{t}\left( \frac{\left( \lambda _{1}t\right)
^{\alpha }}{e^{(\lambda _{1}t)^{\alpha }}-1}+\frac{\left( \lambda t\right)
^{\alpha }}{e^{(\lambda t)^{\alpha }}-1}\right) ,
\end{equation*}%
then%
\begin{equation*}
\phi (t)=\frac{r_{2:2}^{\ast }(t)}{r_{2:2}(t)}=\frac{\frac{(\lambda
_{1}^{\ast }t)^{\alpha }}{e^{(\lambda _{1}^{\ast }t)^{\alpha }}-1}+\frac{%
(\lambda t)^{\alpha }}{e^{(\lambda t)^{\alpha }}-1}}{\frac{(\lambda
_{1}t)^{\alpha }}{e^{(\lambda _{1}t)^{\alpha }}-1}+\frac{(\lambda t)^{\alpha
}}{e^{(\lambda t)^{\alpha }}-1}}.
\end{equation*}%
We have to show that $\phi ^{\prime }(t)\geq 0$ for all $t\geq 0$. The
derivative of $\phi (t)$ is, for $t\geq 0$,
\begin{eqnarray*}
\phi ^{\prime }(t)&\overset{\text{sign}}{=}&\left( \frac{\alpha \lambda
_{1}^{\ast }(\lambda _{1}^{\ast }t)^{\alpha -1}\left( e^{\left( \lambda
_{1}^{\ast }t\right) ^{\alpha }}-1-\left( \lambda _{1}^{\ast }t\right)
^{\alpha }e^{\left( \lambda _{1}^{\ast }t\right) ^{\alpha }}\right) }{\left(
e^{\left( \lambda _{1}^{\ast }t\right) ^{\alpha }}-1\right) ^{2}}+\frac{%
\alpha \lambda (\lambda t)^{\alpha -1}\left( e^{\left( \lambda t\right)
^{\alpha }}-1-\left( \lambda t\right) ^{\alpha }e^{\left( \lambda t\right)
^{\alpha }}\right) }{\left( e^{\left( \lambda t\right) ^{\alpha }}-1\right)
^{2}}\right)  \\
&&\left( \frac{(\lambda _{1}t)^{\alpha }}{e^{(\lambda _{1}t)^{\alpha }}-1}+%
\frac{(\lambda t)^{\alpha }}{e^{(\lambda t)^{\alpha }}-1}\right) -\left(
\frac{(\lambda _{1}^{\ast }t)^{\alpha }}{e^{(\lambda _{1}^{\ast }t)^{\alpha
}}-1}+\frac{(\lambda t)^{\alpha }}{e^{(\lambda t)^{\alpha }}-1}\right)  \\
&&\left( \frac{\alpha \lambda _{1}(\lambda _{1}t)^{\alpha -1}\left(
e^{\left( \lambda _{1}t\right) ^{\alpha }}-1-\left( \lambda _{1}t\right)
^{\alpha }e^{\left( \lambda _{1}t\right) ^{\alpha }}\right) }{\left(
e^{\left( \lambda _{1}t\right) ^{\alpha }}-1\right) ^{2}}+\frac{\alpha
\lambda (\lambda t)^{\alpha -1}\left( e^{\left( \lambda t\right) ^{\alpha
}}-1-\left( \lambda t\right) ^{\alpha }e^{\left( \lambda t\right) ^{\alpha
}}\right) }{\left( e^{\left( \lambda t\right) ^{\alpha }}-1\right) ^{2}}%
\right) .
\end{eqnarray*}%
After some computations, one get that%
\begin{eqnarray*}
\phi ^{\prime }(t)&\overset{\text{sign}}{=}&\frac{(\lambda _{1}^{\ast
}t)^{\alpha }(\lambda _{1}t)^{\alpha }}{\left( e^{(\lambda _{1}^{\ast
}t)^{\alpha }}-1\right) \left( e^{\left( \lambda _{1}t\right) ^{\alpha
}}-1\right) }\left( -\frac{\left( \lambda _{1}^{\ast }t\right) ^{\alpha }}{%
\left( 1-e^{-\left( \lambda _{1}^{\ast }t\right) ^{\alpha }}\right) }+\frac{%
\left( \lambda _{1}t\right) ^{\alpha }}{\left( 1-e^{-\left( \lambda
_{1}t\right) ^{\alpha }}\right) }\right)  \\
&&+\frac{(\lambda t)^{\alpha }(\lambda _{1}^{\ast }t)^{\alpha }}{\left(
e^{(\lambda t)^{\alpha }}-1\right) \left( e^{\left( \lambda _{1}^{\ast
}t\right) ^{\alpha }}-1\right) }\left( -\frac{\left( \lambda _{1}^{\ast
}t\right) ^{\alpha }}{1-e^{-\left( \lambda _{1}^{\ast }t\right) ^{\alpha }}}+%
\frac{\left( \lambda t\right) ^{\alpha }}{1-e^{-\left( \lambda t\right)
^{\alpha }}}\right)  \\
&&+\frac{(\lambda _{1}t)^{\alpha }(\lambda t)^{\alpha }}{\left( e^{(\lambda
_{1}t)^{\alpha }}-1\right) \left( e^{(\lambda t)^{\alpha }}-1\right) }\left(
-\frac{\left( \lambda t\right) ^{\alpha }}{\left( 1-e^{-\left( \lambda
t\right) ^{\alpha }}\right) }+\frac{\left( \lambda _{1}t\right) ^{\alpha }}{%
\left( 1-e^{-\left( \lambda _{1}t\right) ^{\alpha }}\right) }\right)  \\
&&+\frac{(\lambda t)^{\alpha }(\lambda t)^{\alpha }}{\left( e^{(\lambda
t)^{\alpha }}-1\right) \left( e^{\left( \lambda t\right) ^{\alpha
}}-1\right) }\left( -\frac{\left( \lambda t\right) ^{\alpha }}{\left(
1-e^{-\left( \lambda t\right) ^{\alpha }}\right) }+\frac{\left( \lambda
t\right) ^{\alpha }}{\left( 1-e^{-\left( \lambda t\right) ^{\alpha }}\right)
}\right) .
\end{eqnarray*}%
By using the functions defined in (\ref{eclemU}) and (\ref{eclemV}), then
the derivative of $\phi (t)$ can be rewritten by%
\begin{eqnarray*}
\phi ^{\prime }(t) &\overset{\text{sign}}{=}& u(\lambda _{1}^{\ast }t)u(\lambda
_{1}t)\left( -v(\lambda _{1}^{\ast }t)+v(\lambda _{1}t)\right) +u(\lambda
_{1}^{\ast }t)u(\lambda t)\left( -v(\lambda _{1}^{\ast }t)+v(\lambda
t)\right)  \\
&&+u(\lambda t)u(\lambda _{1}t)\left( -v(\lambda t)+v(\lambda _{1}t)\right) .
\end{eqnarray*}%
Note that $u\left( t\right) ,v(t)\geq 0$ for all $t\geq 0$. From Lemmas \ref%
{lemU} and \ref{lemV}, we know that $u\left( t\right) $ is decreasing and $%
v\left( t\right) $ is increasing in $t$. If $\lambda _{1}^{\ast }=\min
(\lambda ,\lambda _{1},\lambda _{1}^{\ast })$ and $\left( \lambda
_{1},\lambda \right) \preceq ^{w}\left( \lambda _{1}^{\ast },\lambda \right)
$, then $\lambda _{1}^{\ast }\leq \lambda \leq \lambda _{1}$ or $\lambda
_{1}^{\ast }\leq \lambda _{1}\leq \lambda $. When $\lambda _{1}^{\ast }\leq
\lambda \leq \lambda _{1}$, we have $\phi ^{\prime }(t)\geq 0$ since $%
v\left( \lambda _{1}^{\ast }t\right) \leq v\left( \lambda t\right) \leq
v\left( \lambda _{1}t\right) $. When $\lambda _{1}^{\ast }\leq \lambda
_{1}\leq \lambda $, we get
\begin{eqnarray*}
\phi ^{\prime }(t) &\geq &u(\lambda t)u(\lambda _{1}t)\left( -v(\lambda
_{1}^{\ast }t)+v(\lambda _{1}t)\right) +u(\lambda _{1}t)u(\lambda t)\left(
-v(\lambda _{1}^{\ast }t)+v(\lambda t)\right)  \\
&&+u(\lambda t)u(\lambda _{1}t)\left( -v(\lambda t)+v(\lambda _{1}t)\right)
\\
&=&2u(\lambda t)u(\lambda _{1}t)\left( -v(\lambda _{1}^{\ast }t)+v(\lambda
_{1}t)\right) \geq 0\text{.}
\end{eqnarray*}%
Therefore $r_{2:2}^{\ast }(t)/r_{2:2}(t)$ is increasing in $t$ for any $%
\alpha >0$.
\end{proof}


\begin{theorem}
\label{th01}Let $X_{\lambda _{1}},X_{\lambda }$ be independent random
variables with $X_{\lambda _{1}}\sim W(\alpha ,\lambda _{1})$ and $%
X_{\lambda }\sim W(\alpha ,\lambda )$ where $\lambda _{1},\lambda >0$. Let $%
Y_{\lambda _{1}^{\ast }},Y_{\lambda }$ be independent random variables with $%
Y_{\lambda _{1}^{\ast }}\sim W(\alpha ,\lambda _{1}^{\ast })$ and $%
Y_{\lambda }\sim W(\alpha ,\lambda )$, where $\lambda _{1}^{\ast },\lambda
>0 $. Suppose $\lambda _{1}^{\ast }=\min (\lambda ,\lambda _{1},\lambda
_{1}^{\ast })$, then for any $\alpha >0$,
\begin{equation*}
\left( \lambda _{1},\lambda \right) \preceq ^{w}\left( \lambda _{1}^{\ast
},\lambda \right) \Rightarrow X_{2:2}\leq _{lr}Y_{2:2}.
\end{equation*}
\end{theorem}

\begin{proof}
From Theorem \ref{th10}, we know that $r_{2:2}^{\ast }(t)/r_{2:2}(t)$ is
increasing in $t$ under the given assumptions. Since $\left( \lambda
_{1},\lambda \right) \preceq ^{w}\left( \lambda _{1}^{\ast },\lambda \right)
$, it follows from Theorem \ref{th07} that $X_{2:2}\leq _{rh}Y_{2:2}$. Thus
the required result follows from Theorem 1.C.4 in \cite{Shaked:2007}.
\end{proof}

\vskip .5em

Note that when $\left( \lambda _{1},\lambda \right) \preceq ^{w}\left(
\lambda _{1}^{\ast },\lambda \right) $ we have the following three
possibilities:%
\begin{equation*}
\lambda _{1}^{\ast }\leq \lambda \leq \lambda _{1},\text{ }\lambda
_{1}^{\ast }\leq \lambda _{1}\leq \lambda \text{ or }\lambda \leq \lambda
_{1}^{\ast }\leq \lambda _{1}\text{.}
\end{equation*}%
The two first are included in assumption of Theorem \ref{th01} and so a
natural question is whether this theorem holds for $\lambda \leq \lambda
_{1}^{\ast }\leq \lambda _{1}$. The following example gives a negative
answer.

\begin{example}
Let $\left( X_{\lambda _{1}},X_{\lambda }\right) $ be a vector of
heterogeneous Weibull random variables with $\alpha =0.3$ and scale
parameters $\lambda =0.2$ and $\lambda _{1}=3.5$. Let $\left( Y_{\lambda
_{1}^{\ast }},Y_{\lambda }\right) $ be a vector of heterogeneous Weibull
random variables with $\alpha =0.3$ and scale parameters $\lambda =0.2$ and $%
\lambda _{1}^{\ast }=2$. Obviously $\left( \lambda _{1},\lambda \right)
\preceq ^{w}\left( \lambda _{1}^{\ast },\lambda \right) $ and $\lambda \leq
\lambda _{1}^{\ast }\leq \lambda _{1}$. However $X_{2:2}\nleq _{lr}Y_{2:2}$
since $f_{2:2}^{\ast }(t)/f_{2:2}(t)$ is not increasing in $t$ as it can be
seen from Figure \ref{figcocpdfWeib-03}. Analogously, from Figure \ref%
{figcocpdfWeib-13}, it can be seen that $X_{2:2}\nleq _{lr}Y_{2:2}$ when $%
\alpha =1.3$.
\end{example}

\begin{figure}[!ht]
\centering
\includegraphics[scale=0.7, keepaspectratio]{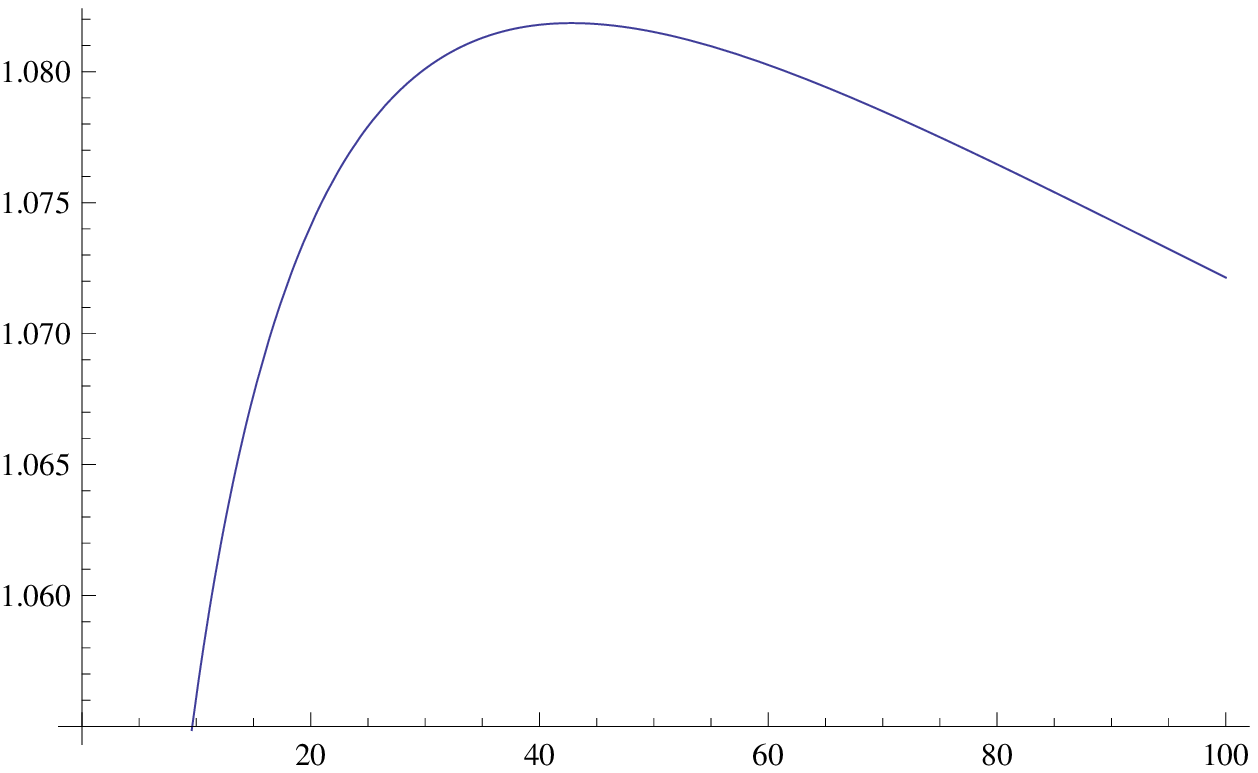}
\caption{Plot of the $f_{2:2}^{* }(t)/f_{2:2}(t)$ when $\protect\alpha =0.3$%
, $\protect\lambda^{*} =\left( 0.2,2\right) $ and $\protect\lambda =\left(
0.2,3.5\right) $ for random variables with Weibull distributions}
\label{figcocpdfWeib-03}
\end{figure}
\begin{figure}[!ht]
\centering
\includegraphics[scale=0.7, keepaspectratio]{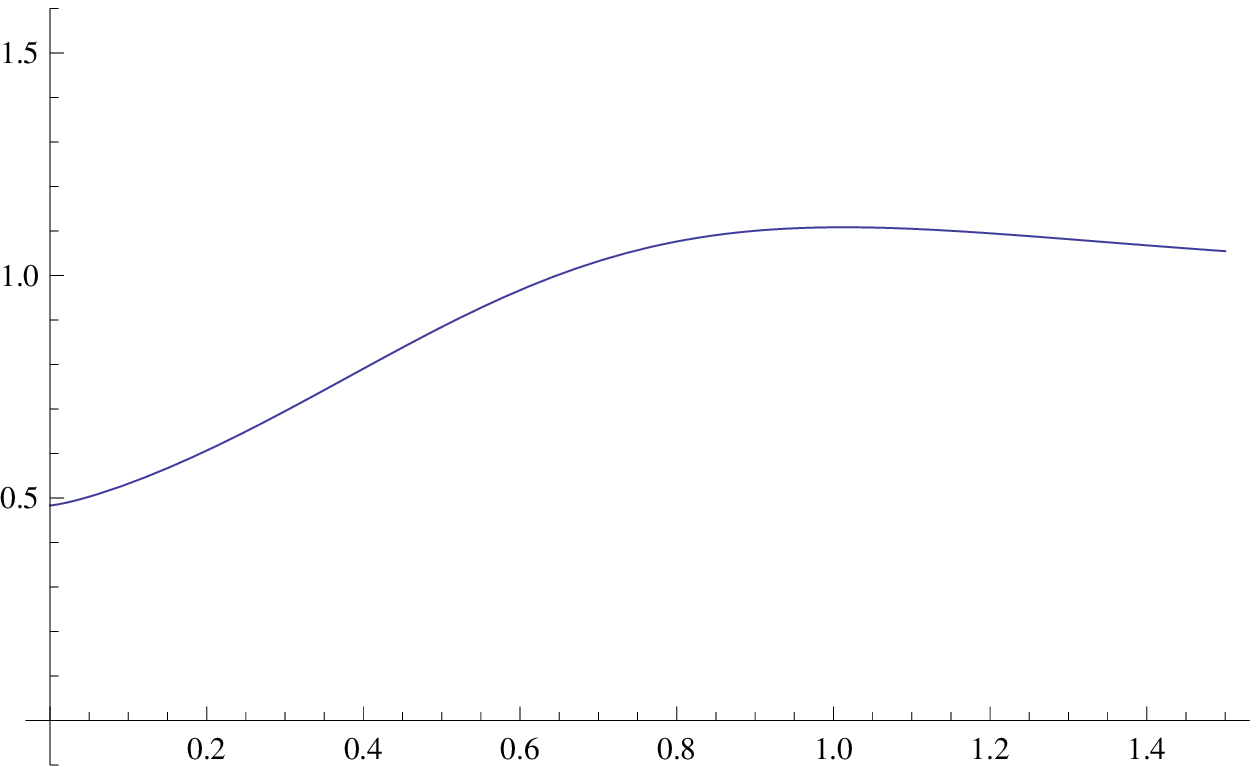}
\caption{Plot of the $f_{2:2}^{* }(t)/f_{2:2}(t)$ when $\protect\alpha =1.3$%
, $\protect\lambda^{*} =\left( 0.2,2\right) $ and $\protect\lambda =\left(
0.2,3.5\right) $ for random variables with Weibull distributions}
\label{figcocpdfWeib-13}
\end{figure}

From Theorems \ref{th01} and \ref{th09}, the following result can be proved
using arguments similar to those used in Theorem 3.6 of \cite{Zhao:2011}.

\begin{theorem}
\label{th11} Let $X_{\lambda _{1}},X_{\lambda _{2}}$ be independent random
variables with $X_{\lambda _{i}}\sim W(\alpha ,\lambda _{i})$ where $\lambda
_{i}>0$, $i=1,2$, and let $X_{\theta _{1}},X_{\theta _{2}}$ be another set
of independent random variables with $X_{\theta _{i}}\sim W(\alpha ,\theta
_{i})$ where $\theta _{i}>0$, $i=1,2$. Suppose $\theta _{1}\leq \lambda
_{1}\leq \lambda _{2}\leq \theta _{2}$. Then $\left( \lambda _{1},\lambda
_{2}\right) \preceq ^{w}\left( \theta _{1},\theta _{2}\right) $ implies that
$X_{2:2}^{\lambda }\leq _{lr}X_{2:2}^{\theta }$, for $0<\alpha \leq 1$.
\end{theorem}

Note that Theorem 2.4 in \cite{Joo:2010} for two heterogeneous exponential
random variables can be seen as a particular case of Theorem \ref{th11},
since likelihood ratio order implies hazard rate order and exponential
distributions are a particular case of Weibull distributions.


Next we extend the study of likelihood ratio ordering among parallel systems
from the two-variable case to multiple-outlier Weibull models.

\begin{theorem}
\label{th17}Let $X_{1},\ldots ,X_{n}$ be independent random variables
following the multiple-outlier Weibull model such that $X_{i}\sim W(\alpha
,\lambda _{1})$ for $i=1,\ldots ,p$\ and $X_{j}\sim W(\alpha ,\lambda )$ for
$j=p+1,\ldots ,n$, with $\lambda _{1},\lambda >0$. Let $Y_{1},\ldots ,Y_{n}$
be another set of independent random variables following the
multiple-outlier Weibull model with $Y_{i}\sim W(\alpha ,\lambda _{1}^{\ast
})$ for $i=1,\ldots ,p$ and $Y_{j}\sim W(\alpha ,\lambda )$ for $%
j=p+1,\ldots ,n$, with $\lambda _{1}^{\ast },\lambda >0$. Suppose $\lambda
_{1}^{\ast }=\min (\lambda ,\lambda _{1},\lambda _{1}^{\ast })$, then for
any $\alpha >0$,
\begin{equation*}
(\underbrace{\lambda _{1},\ldots ,\lambda _{1}}_{p},\underbrace{\lambda,\ldots ,\lambda}_{q})
\preceq ^{w}
(\underbrace{\lambda _{1}^{\ast },\ldots ,\lambda _{1}^{\ast }}_{p},\underbrace{\lambda,\ldots ,\lambda}_{q})
\Rightarrow \frac{r_{n:n}^{\ast }(t)}{r_{n:n}(t)}%
\text{ is increasing in }t,
\end{equation*}
where $p+q=n$.
\end{theorem}

\begin{proof}
From (\ref{ecrh}) we have the reverse hazard rate function of $X_{n:n}$
\begin{equation*}
r_{n:n}(t)=\frac{\alpha }{t}\left( p\frac{\left( \lambda _{1}t\right)
^{\alpha }}{e^{(\lambda _{1}t)^{\alpha }}-1}+q\frac{\left( \lambda t\right)
^{\alpha }}{e^{(\lambda t)^{\alpha }}-1}\right) ,
\end{equation*}%
where $p+q=n$. Let%
\begin{equation*}
\phi (t)=\frac{r_{n:n}^{\ast }(t)}{r_{n:n}(t)}=\frac{p\frac{(\lambda
_{1}^{\ast }t)^{\alpha }}{e^{(\lambda _{1}^{\ast }t)^{\alpha }}-1}+q\frac{%
(\lambda t)^{\alpha }}{e^{(\lambda t)^{\alpha }}-1}}{p\frac{(\lambda
_{1}t)^{\alpha }}{e^{(\lambda _{1}t)^{\alpha }}-1}+q\frac{(\lambda
t)^{\alpha }}{e^{(\lambda t)^{\alpha }}-1}}.
\end{equation*}%
As in the proof of Theorem \ref{th10}, for $t\geq 0$, the derivative of $%
\phi (t)$ can be rewritten by%
\begin{eqnarray*}
\phi ^{\prime }(t)&\overset{\text{sign}}{=}& p^{2}u(\lambda _{1}^{\ast
}t)u(\lambda _{1}t)\left( -v(\lambda _{1}^{\ast }t)+v(\lambda _{1}t)\right)
+pqu(\lambda _{1}^{\ast }t)u(\lambda t)\left( -v(\lambda _{1}^{\ast
}t)+v(\lambda t)\right)  \\
&&+pqu(\lambda t)u(\lambda _{1}t)\left( -v(\lambda t)+v(\lambda
_{1}t)\right) .
\end{eqnarray*}%
Note that $u\left( t\right) ,v(t)\geq 0$ for all $t\geq 0$. From Lemmas \ref%
{lemU} and \ref{lemV}, we know that $u\left( t\right) $ is decreasing and $%
v\left( t\right) $ is increasing in $t$. If $\lambda _{1}^{\ast }=\min
(\lambda ,\lambda _{1},\lambda _{1}^{\ast })$ and $\left( \lambda
_{1},\lambda \right) \preceq ^{w}\left( \lambda _{1}^{\ast },\lambda \right)
$, then $\lambda _{1}^{\ast }\leq \lambda \leq \lambda _{1}$ or $\lambda
_{1}^{\ast }\leq \lambda _{1}\leq \lambda $. When $\lambda _{1}^{\ast }\leq
\lambda \leq \lambda _{1}$, we have $\phi ^{\prime }(t)\geq 0$ since $%
v\left( \lambda _{1}^{\ast }t\right) \leq v\left( \lambda t\right) \leq
v\left( \lambda _{1}t\right) $. When $\lambda _{1}^{\ast }\leq \lambda
_{1}\leq \lambda $, we get
\begin{eqnarray*}
\phi ^{\prime }(t) &\geq &p^{2}u(\lambda t)u(\lambda _{1}t)\left( -v(\lambda
_{1}^{\ast }t)+v(\lambda _{1}t)\right) +pqu(\lambda _{1}t)u(\lambda t)\left(
-v(\lambda _{1}^{\ast }t)+v(\lambda t)\right)  \\
&&+pqu(\lambda t)u(\lambda _{1}t)\left( -v(\lambda t)+v(\lambda
_{1}t)\right)  \\
&=&npu(\lambda t)u(\lambda _{1}t)\left( -v(\lambda _{1}^{\ast }t)+v(\lambda
_{1}t)\right) \geq 0\text{.}
\end{eqnarray*}%
Therefore $r_{n:n}^{\ast }(t)/r_{n:n}(t)$ is increasing in $t$ for any $%
\alpha >0$.
\end{proof}


\begin{theorem}
\label{th14}Let $X_{1},\ldots ,X_{n}$ be independent random variables such
that $X_{i}\sim W(\alpha ,\lambda _{1})$ for $i=1,\ldots ,p$\ and $X_{j}\sim
W(\alpha ,\lambda )$ for $j=p+1,\ldots ,n$, with $\lambda _{1},\lambda >0$.
Let $Y_{1},\ldots ,Y_{n}$ be independent nonnegative random variables with $%
Y_{i}\sim W(\alpha ,\lambda _{1}^{\ast })$ for $i=1,\ldots ,p$ and $%
Y_{j}\sim W(\alpha ,\lambda )$ for $j=p+1,\ldots ,n$, with $\lambda
_{1}^{\ast },\lambda >0$. Suppose $\lambda _{1}^{\ast }=\min (\lambda
,\lambda _{1},\lambda _{1}^{\ast })$, then%
\begin{equation*}
(\underbrace{\lambda _{1},\ldots ,\lambda _{1}}_{p},\underbrace{\lambda,\ldots ,\lambda}_{q})
\preceq ^{w}
(\underbrace{\lambda _{1}^{\ast },\ldots ,\lambda _{1}^{\ast }}_{p},\underbrace{\lambda,\ldots ,\lambda}_{q})
\Rightarrow X_{n:n}\leq _{lr}Y_{n:n},
\end{equation*}
for $p+q=n$ and any $\alpha >0$.
\end{theorem}

\begin{proof}
From Theorem \ref{th17}, we know that under the given conditions,$%
r_{n:n}^{\ast }(t)/r_{n:n}(t)$ is increasing in $t$. Since $\left( \lambda
_{1},\ldots ,\lambda _{1},\lambda ,\ldots ,\lambda \right) \preceq
^{w}\left( \lambda _{1}^{\ast },\ldots ,\lambda _{1}^{\ast },\lambda ,\ldots
,\lambda \right) $, then $X_{n:n}\leq _{rh}Y_{n:n}$ from Theorem \ref{th07}.
Thus the required result follows from Theorem 1.C.4 in \cite{Shaked:2007}.
\end{proof}

This result is similar to Theorem 4.7 without any restriction on the common
shape parameter with but with an additional constraint on the scale
parameters.

\section*{Acknowledgements}

The research of N.T. was supported by the Portuguese Government through the
Funda\c{c}\~{a}o para a Ci\^{e}ncia e Tecnologia (FCT) under the grant
SFRH/BPD/91832/2012 and partially supported by the Centro de Matem\'{a}tica
da Universidade de Coimbra (CMUC), funded by the European Regional
Development Fund through the program COMPETE and by the Portuguese
Government through the FCT - Funda\c{c}\~{a}o para a Ci\^{e}ncia e a
Tecnologia under the project PEst-C/MAT/UI0324/2013.
The authors are thankful to the Editor and the referee
for their constructive comments and suggestions which have improved the
presentation of the paper.


\end{document}